\documentclass[reqno,12pt,letterpaper]{amsart}
\usepackage{amsmath,amssymb,amsthm,graphicx,mathrsfs,url}
\usepackage[usenames,dvipsnames]{color}
\usepackage[colorlinks=true,linkcolor=Red,citecolor=Green]{hyperref}

\def\?[#1]{\textbf{[#1]}\marginpar{\Large{\textbf{??}}}}
\newtheorem{prop}{Proposition}
\newtheorem{thm}[prop]{Theorem}
\newtheorem{defi}[prop]{Definition}
\newtheorem{exa}[prop]{Example}
\newtheorem{lem}[prop]{Lemma}

\numberwithin{equation}{section}
\numberwithin{prop}{section}

\newcommand{\dbar}{\bar{\partial}}
\renewcommand{\Re}{\mathop{\rm Re}\nolimits}
\renewcommand{\Im}{\mathop{\rm Im}\nolimits}

\DeclareMathOperator{\comp}{comp}
\DeclareMathOperator{\loc}{loc}
\DeclareMathOperator{\pv}{p.v.}
\DeclareMathOperator{\supp}{supp}
\DeclareMathOperator{\sgn}{sgn}
\DeclareMathOperator{\Sp}{Sp}

\begin{document}
\title{Fractal uncertainty principle with explicit exponent}

\author{Long Jin}
\email{ljin@math.tsinghua.edu.cn}
\address{Yau Mathematical Sciences Center, Tsinghua University, Beijing, China}
\author{Ruixiang Zhang}
\email{rzhang@math.ias.edu}
\address{Institute for Advanced Study, 1 Einstein Drive, Simonyi Hall 105, Princeton, NJ 08540}

\begin{abstract}
We prove an explicit formula for the dependence of the exponent $\beta$ in the fractal uncertainty principle of Bourgain--Dyatlov \cite{fullgap} on the dimension $\delta$ and on the regularity constant $C_R$ for the regular set. In particular, this implies an explicit essential spectral gap for convex co-compact hyperbolic surfaces when the Hausdorff dimension of the limit set is close to 1.
\end{abstract}

\maketitle


\section{Introduction}
\label{s:intro}
In this paper, we study the fractal uncertainty principle for the Fourier transform. Roughly speaking, it states that a function and its Fourier transform cannot simultaneously concentrate near fractal sets. More precisely, the fractal sets we study satisfy a condition called \emph{Ahlfors-David regularity} and we refer to Section \ref{s:regular} for the definition of a $\delta$-regular set. In a recent paper, Bourgain-Dyatlov \cite{fullgap} proved the following remarkable theorem.


\begin{thm}
\label{t:bdfup}
Let $0<\delta<1$, $C_R\geq1$, $N\geq1$ and assume that 

(a) $X\subset[-1,1]$ is $\delta$-regular with constant $C_R$ on scales $N^{-1}$ to 1, and 

(b) $Y\subset[-N,N]$ is $\delta$-regular with constant $C_R$ on scales 1 to $N$. 

Then there exist $\beta>0$, $C$ depending only on $\delta$, $C_R$ such that for all $f\in L^2(\mathbb{R})$,
\begin{equation}
\supp\widehat{f}\subset Y\quad\Rightarrow\quad \|f\|_{L^2(X)}\leq CN^{-\beta}\|f\|_{L^2(\mathbb{R})}.
\end{equation}
\end{thm}

The goal of this paper is to give an effective fractal uncertainty principle exponent $\beta$ in Theorem \ref{t:bdfup} with dependence on $\delta$ and $C_R$. Here is our main result.


\begin{thm}
\label{t:effup}
Theorem \ref{t:bdfup} holds for 
\begin{equation}
\label{e:fupexp}
\beta=\exp\left[-\exp\left(K(C_R\delta^{-1}(1-\delta)^{-1})^{K(1-\delta)^{-2}}\right)\right].
\end{equation}
Here $K>0$ is a universal constant.
\end{thm}


\subsection{Fractal uncertainty principle}
\label{s:app}

Theorem \ref{t:bdfup} implies a fractal uncertainty principle for general Fourier integral operators (see Theorem \ref{t:generalfup}, also \cite[\S 4.2]{fullgap}) which is useful in the study of quantum chaos. The fractal uncertainty principle was first introduced by Dyatlov--Zahl \cite{fup} to study the essential spectral gap for resonances of convex co-compact hyperbolic surfaces and later by Dyatlov--Jin \cite{oqm} in a simple model called open quantum baker's maps. In both situations, the fractal uncertainty principle implies the existence of an essential spectral gap for systems with trapped sets of any dimension. For example, we now have the following application of Theorem \ref{t:effup} for convex co-compact hyperbolic surfaces.

\begin{thm}
Let $M=\Gamma\setminus\mathbb{H}^2$ be a convex co-compact hyperbolic surface with limit set $\Lambda_\Gamma$. Let $\delta$ and $C_R$ be the Hausdorff dimension and the regularity constant of $\Lambda_\Gamma$, respectively. Then $M$ has an essential spectral gap of size 
\begin{equation*}
\beta_M=\exp\left[-\exp\left(K(C_R\delta^{-1}(1-\delta)^{-1})^{K(1-\delta)^{-3}}\right)\right],
\end{equation*}
i.e. there are only finitely many resonances for the Laplacian on $M$ when $\Im\lambda>-\beta_M$.
\end{thm}

We refer to Theorem \ref{t:gap-hyper} for a more precise statement and Theorem \ref{t:gap-improves} for a similar result on open quantum baker's maps. It is conjectured by Zworski \cite{zworski} that there is always an essential spectral gap whenever the trapped set is hyperbolic. 

In a recent work of Dyatlov--Jin \cite{messupp}, the fractal uncertainty principle is also applied to closed quantum chaotic systems, to understand the behavior of Laplace eigenfunctions on compact hyperbolic surfaces in the semiclassical limit. In particular, it is shown that any semiclassical measure must have full support on the unit cosphere bundle. The \emph{Quantum Unique Ergodicity Conjecture} of Rudnick--Sarnak \cite{que} states that the Liouville measure is the only semiclassical measure. For any subset $U$ of the unit cosphere bundle, combining Theorem \ref{t:effup} with the argument in \cite{messupp} should give a lower bound on the semiclassical measure of $U$, in terms only of the geometry of $U$. But the bound obtained in this way seems unlikely to be even comparable to the Liouville measure.

We would like also to mention a new paper of Han--Schlag \cite{hdfup} appeared after the initial version of our paper, in which they proved a higher dimensional version of fractal uncertainty principle (Theorem \ref{t:bdfup}) with a stronger condition on one of the set $X$ and $Y$. We refer to their paper for more discussion on this issue.  


\subsection{Beurling--Malliavin multiplier theorem and the proof}
\label{s:bmm}



The key ingredient of the proof of Theorem \ref{t:bdfup} in \cite{fullgap} is the  Beurling--Malliavin multiplier theorem \cite{bmm}. There are many other different formulations and proofs in the literature for this theorem, for example, \cite{malliavin}, \cite{koosis} and \cite{bm7}. We refer to Mashreghi--Nazarov--Havin \cite{bm7} for a survey of the history of this theorem and more references. Here we follow the formulation \cite[Theorem BM1]{bm7}.

\begin{thm}
\label{t:bmm}
Let $\omega\in C^1(\mathbb{R};(0,1])$ satisfy the conditions
\begin{equation*}
\int_\mathbb{R}\frac{|\log\omega(\xi)|}{1+\xi^2}d\xi<\infty
\end{equation*}
\begin{equation*}
\sup|\partial_\xi\log\omega|<\infty.
\end{equation*}
Then for each $c_0>0$, there exists a function $\psi\in L^2(\mathbb{R})$ such that 
\begin{equation*}
\supp\psi\subset[-c_0,c_0],\quad |\widehat{\psi}|\leq\omega, \quad\psi\not\equiv0.
\end{equation*}
\end{thm}


In their proof of Theorem \ref{t:bdfup}, Bourgain--Dyatlov use a quantitative refinement of Theorem \ref{t:bmm}, which follows from Theorem \ref{t:bmm} itself.

\begin{thm}
\label{t:bdbmm}
Let $C_0,c_0>0$ and $\omega\in C^1(\mathbb{R};(0,1])$ satisfy the conditions
\begin{equation*}
\int_\mathbb{R}\frac{|\log\omega(\xi)|}{1+\xi^2}d\xi\leq C_0
\end{equation*}
\begin{equation}
\label{e:lipschitz}
\sup|\partial_\xi\log\omega|\leq C_0.
\end{equation}
Then there exists $c>0$ depending only on $C_0,c_0$ and a function $\psi\in L^2(\mathbb{R})$ such that 
\begin{equation*}
\supp\psi\subset[-c_0,c_0],\quad |\widehat{\psi}|\leq\omega^c, \quad\|\widehat{\psi}\|_{L^2(-1,1)}\geq c.
\end{equation*}
\end{thm}

However, the proof in \cite{fullgap} is done by contradiction, thus no effective dependence of $c$ on $C_0$ and $c_0$ is obtained. In our approach, we use a different version of the multiplier theorem which appeared in Mashreghi--Nazarov--Havin \cite{bm7} to obtain a variant, but effective version of Theorem \ref{t:bdbmm}, see Theorem \ref{t:effect-multiplier} below. In particular, we replace the Lipschitz condition \eqref{e:lipschitz} on $\log\omega$ by a Lipschitz condition on the Hilbert transform of $\log\omega$, see \eqref{e:hilbert-bd}. This weaker version is only the first step in the ``seventh'' proof of Beurling--Malliavin theorem given by Mashreghi--Mazarov--Havin \cite{bm7}. Nevertheless, it is enough for our applications and in fact, the Hilbert transform condition is more natural in the context of multiplier theorems for the Cartwright class of entire functions, due to its natural relation to the Cauchy integrals. This was explained to us by Alexander Sodin during the workshop at IAS and we are very thankful to him. Also, Jean Bourgain \cite{bourgain} informed us another proof of a multiplier theorem with similar Lipschitz conditions using H\"{o}rmander's $\dbar$-theorem. We refer to \cite[\S 3.1]{bm7} and \cite[Part 2 \S 3.5]{up} for detailed discussion about the background of the Beurling--Malliavin mulitplier theorem.

With the quantitative multiplier Theorem \ref{t:effect-multiplier}, we can proceed exactly as in Bourgain--Dyatlov \cite{fullgap} but with careful treatment of the constants. The only essential difference is checking that the multiplier adapted to regular sets constructed in Bourgain--Dyatlov also satisfies the new condition \eqref{e:hilbert-bd} instead of \eqref{e:lipschitz}.


The paper is organized as follows. In \S \ref{s:prelim}, we review some preliminary results. In particular, we discuss the basic properties of the Hilbert transform (\S \ref{s:hilbert}) which plays a great role in our argument. In \S \ref{s:effbmm}, we establish the effective version of Beurling-Malliavin theorem \ref{t:effect-multiplier} which is the new ingredient needed to make Theorem \ref{t:bdfup} effective. In \S \ref{s:proof}, we follow the argument of Bourgain-Dyatlov \cite{fullgap} and carefully track all the constants. Finally, we discuss in \S \ref{s:application} the fractal uncertainty principle for general Fourier integral operators and applications to spectral gaps for convex co-compact hyperbolic surfaces and open quantum baker's maps. 



\subsection*{Acknowledgements}
The authors are very grateful to Semyon Dyatlov for encouragement and helpful discussions and to Kiril Datchev for many useful suggestions. We are especially thankful to Jean Bourgain for sharing with us some preliminary notes \cite{bourgain} about the multiplier theorem and H\"{o}rmander's $\dbar$-theorem, and to Alexander Sodin for explaining to us the background of the multiplier theorems in the context of complex analysis during the Emerging Topics Working Groups at IAS in fall 2017. We are also indebted to Rui Han and Wilhelm Schlag for pointing out a mistake in the initial version where we miss an additional exponential in the computation. We would also like to thank BICMR at Peking University for hospitality during our visit in May 2017 where we started the project, and IAS for hosting a week-long workshop on the topic of quantum chaos and fractal uncertainty principle. LJ is also thankful to Sun Yat-Sen University and YMSC at Tsinghua University where part of the project was finished. RZ is supported by NSF under Grant No. 1638352 and the James D. Wolfensohn Fund.


\section{Preliminaries}
\label{s:prelim}


\subsection{Notation}

We start with some standard notation that we use in this paper. First, we denote the Fourier transform on $\mathbb{R}$ by
\begin{equation*}
\widehat{f}(\xi)=\mathcal{F}f(\xi)=\int_{\mathbb{R}}e^{-2\pi ix\xi}f(x)dx,
\end{equation*}
then its inverse is given by
\begin{equation*}
\mathcal{F}^{-1}g(x)=\int_{\mathbb{R}}e^{2\pi ix\xi}g(\xi)d\xi.
\end{equation*}

Also, for $x\in\mathbb{R}$, we write $\langle x\rangle=(1+x^2)^{1/2}$
and $\lfloor x\rfloor$ denotes the largest integer not exceeding $x$ while $\lceil x\rceil$ denotes the smallest integer greater or equal to $x$. For a subset $U$ of $\mathbb{R}$, we use $1_U$ to denote the indicator function of $U$.

In this paper, we always use $K$ to denote a universal constant which might change from line to line and $C$ as well as $C_1,C_2,\ldots$ to denote constants that may depend on $C_R$ and $\delta$.


\subsection{Ahlfors--David regular sets}
\label{s:regular}

We recall the definition and properties of Ahlfors--David regular sets from \cite{fullgap}.
\begin{defi}
Let $X\subset\mathbb{R}$ be a nonempty closed set and $\delta\in[0,1]$, $C_R\geq1$, $0\leq\alpha_0\leq\alpha_1\leq\infty$. We say that $X$ is $\delta$-regular with constant $C_R$ on scales $\alpha_0$ to $\alpha_1$ if there exists a Borel measure $\mu_X$ on $\mathbb{R}$ such that 

(a) $\mu_X$ is supported on $X$, i.e. $\mu_X(\mathbb{R}\setminus X)=0$; 

(b) for each interval $I$ of size $|I|\in[\alpha_0,\alpha_1]$, we have $\mu_X(I)\leq C_R|I|^\delta$; 

(c) if additionally $I$ is centered at a point in $X$, then $\mu_X(I)\geq C_R^{-1}|I|^\delta$.
\end{defi}

The most important property of regular sets for our argument is the following small covering property (see \cite[Lemma 2.8]{fullgap}).

\begin{prop}
\label{p:smallcover}
Let $X$ be a $\delta$-regular set with constant $C_R$ on scales $\alpha_0$ to $\alpha_1$. Let $I$ be an interval and $\rho>0$ satisfy $\alpha_0\leq\rho\leq|I|\leq\alpha_1$. Then there exists a nonoverlapping collection $\mathcal{J}$ of $N_\mathcal{J}\leq 12C_R^2(|I|/\rho)^{\delta}$ intervals of size $\rho$ each such that $X\cap I\subset\bigcup_{J\in\mathcal{J}}J$. In fact, we can arrange that each elements in the $\mathcal{J}$ is of the form $\rho[j,j+1]$ with $j\in\mathbb{Z}$.
\end{prop}

In other words, a regular set must be ``porous'', i.e. there are lots of intervals that are missing from a regular sets. We refer the reader to \cite[\S 2]{fullgap} for other basic properties of regular sets.


\subsection{Hilbert transform}
\label{s:hilbert}


The Hilbert transform plays an important role in the proof of the Beurling--Malliavin multiplier theorem in \cite{bm7}. Here we follow the presentation there to give a brief review of some basic properties.


Let $\mathcal{H}_0$ be the standard Hilbert transform defined as convolution with $\pv\frac{1}{\pi x}$: For $f\in C_0^\infty(\mathbb{R})$, (or more generally, $f\in L^1(\mathbb{R},\langle x\rangle^{-1}dx)$)
\begin{equation}
\label{e:def-hilbert}
\mathcal{H}_0(f)(x)=f\ast\pv\frac{1}{\pi x}=\frac{1}{\pi}\pv\int\frac{f(x-t)}{t}dt:=\frac{1}{\pi}\lim_{\varepsilon\to0+}\int_{|t|\geq\varepsilon}\frac{f(x-t)}{t}dt.
\end{equation}
In particular, if $f$ is continuously differentiable, then
\begin{equation}
\label{e:c1-hilbert}
\mathcal{H}_0(f)(x)=\frac{1}{\pi}\int_{|t|\leq 1}\frac{f(x-t)-f(x)}{t}dt+\int_{|t|>1}\frac{f(x-t)}{t}dt.
\end{equation}
Moreover, \eqref{e:def-hilbert} converges for almost all $x$ if $f\in L^1(\mathbb{R},\langle x\rangle^{-1}dx)$.


An equivalent definition using Fourier transform is given by
\begin{equation}
\label{e:def-hilbert-alt}
\mathcal{H}_0(f)(x)=\mathcal{F}^{-1}(-i\sgn\xi\cdot\widehat{f}(\xi)),\quad f\in C_0^\infty(\mathbb{R}).
\end{equation}
This definition naturally extends to $L^2(\mathbb{R})$ which is a subspace of $L^1(\mathbb{R},\langle x\rangle^{-1}dx)$. Moreover it is straightforward from \eqref{e:def-hilbert-alt} that $\mathcal{H}_0$ is a unitary operator in $L^2(\mathbb{R})$ and satisfies the inversion formula 
\begin{equation}
\label{e:inversion-0}
\mathcal{H}_0(\mathcal{H}_0(f))=-f,
\end{equation}
for all $f\in L^2(\mathbb{R})$. The standard approximation argument can be used to show that \eqref{e:inversion-0} is true for all $f\in L^1(\mathbb{R},\langle x\rangle^{-1}dx)$ with $\mathcal{H}_0(f)\in L^1(\mathbb{R},\langle x\rangle^{-1}dx)$.


The Hilbert transform $\mathcal{H}_0$ commutes with translations and dilations. Let $T_{x_0}$ and $D_\lambda$ be the translation and dilation operators:
\begin{equation*}
\begin{split}
(T_{x_0}f)(x)=f(x+x_0), \quad x_0\in\mathbb{R};\\
(D_\lambda f)(x)=f(\lambda x), \quad \lambda>0.
\end{split}
\end{equation*}
Then we have
\begin{equation}
\label{e:hilbert-scale}
\mathcal{H}_0(T_{x_0}(f))=T_{x_0}(\mathcal{H}_0(f)),
\quad
\mathcal{H}_0(D_\lambda(f))=D_\lambda(\mathcal{H}_0(f)).
\end{equation}


Moreover, for any $f\in C_0^\infty$, we have $\mathcal{H}_0(f)\in C^\infty$ with 
$\mathcal{H}_0(f)'=\mathcal{H}_0(f')$ satisfying the bound 
\begin{equation}
\label{e:hilbert-deriative-bound}
\mathcal{H}_0(f')(x)=\mathcal{O}(\langle x\rangle^{-2}).
\end{equation}


For our applications, we are mostly interested in Hilbert transforms of smooth functions. However, we also want to include functions not in $L^1(\mathbb{R},\langle x\rangle^{-1}dx)$, for example, the constant functions, or even some unbounded functions (see the examples later). Therefore, we want to ``extend'' the Hilbert transform to an even larger space $L^1(\mathbb{R},\langle x\rangle^{-2}dx)$ which, in particular, contains the space $L^\infty(\mathbb{R})$ as well as functions that only grow like $|x|^{1-\epsilon}$ as $|x|\to\infty$. To achieve this, we modify the integral kernel so that it decays like $|x|^{-2}$ as $|x|\to\infty$:
\begin{equation}
\label{e:hilbert}
\mathcal{H}(f)(x)=\frac{1}{\pi}\pv\int_{\mathbb{R}}f(t)\left(\frac{1}{x-t}+\frac{t}{t^2+1}\right)dt,
\quad f\in L^1(\mathbb{R},\langle x\rangle^{-2}dx).
\end{equation}
Again, for $f\in C_0^\infty$ or continuously differentiable, we have a similar expression to \eqref{e:c1-hilbert} and in general, the formula \eqref{e:hilbert} converges almost everywhere when $f\in L^1(\mathbb{R},\langle x\rangle^{-2}dx)$.


If $f\in L^1(\mathbb{R},\langle x\rangle^{-1}dx)$, then definition differs from the standard Hilbert transform \eqref{e:def-hilbert} by a constant 
$c(f)=\int_{\mathbb{R}}\frac{tf(t)}{t^2+1}dt$.
Moreover, since the Hilbert transform $\mathcal{H}$ of a constant function is zero, we have the inversion formula up to constants: say for $f\in L^2(\mathbb{R})$,
\begin{equation}
\label{e:inversion}
\mathcal{H}(\mathcal{H}(f))=-f+c(\mathcal{H}_0(f)).
\end{equation}
Again, we can extend \eqref{e:inversion} to all functions $f\in L^1(\mathbb{R},\langle x\rangle^{-2}dx)$ with $\mathcal{H}(f)\in L^1(\mathbb{R},\langle x\rangle^{-2}dx)$ by replacing $c(\mathcal{H}_0(f))$ with a constant depending on $f$.


Later we are mainly concerned with the Lipschitz norm of the Hilbert transform of a function. Assume that $f\in C^\infty(\mathbb{R})$ with $\langle x\rangle^{-1}f(x)\to0$ as $|x|\to\infty$ and $f'\in L^1(\mathbb{R},\langle x\rangle^{-1}dx)$. In particular, such $f$ are in $L^1(\mathbb{R},\langle x\rangle^{-2}dx)$ and by direct computation we have
\begin{equation*}
\mathcal{H}(f)'=\mathcal{H}_0(f'). 
\end{equation*}

Here are some examples that we will use later in the proof.

\begin{exa}
Let $f(x)=\log(x^2+1)$, then 
\begin{equation*}
f'(x)=\frac{2x}{x^2+1}
\end{equation*}
and we can compute
\begin{equation}
\label{e:example1}
\mathcal{H}(f)'(x)=\mathcal{H}_0(f')(x)=-\frac{2}{x^2+1}.
\end{equation}
either by explicit integration or rewriting $f'(x)=(x+i)^{-1}+(x-i)^{-1}$ and use the definition \eqref{e:def-hilbert-alt}.
\end{exa}

\begin{exa}
Let $f(x)=\langle x\rangle^{1/2}$, then $f'(x)=\frac{1}{2}x\langle x\rangle^{-3/2}$ and we can estimate the $L^\infty$ norm of 
\begin{equation*}
\mathcal{H}(f)'(x)=\mathcal{H}_0(f')=\pv\int_{\mathbb{R}}\frac{f'(x-t)}{t}dt=I_1+I_2.
\end{equation*}
where
\begin{equation*}
I_1=\int_{-1}^1\frac{f'(x-t)-f'(x)}{t}dt
\end{equation*}
and 
\begin{equation*}
I_2=\int_{|t|\geq 1}\frac{f'(x-t)}{t}dt.
\end{equation*}
Since $f''=\frac{1}{2}\langle x\rangle^{-3/2}-\frac{3}{4}x^2\langle x\rangle^{-7/2}$, we see that $\|f''\|_{L^\infty}\leq \frac{3}{4}$ and thus $|I_1|\leq\frac{3}{2}$. For $I_2$, we have
\begin{equation*}
|I_2|\leq\int_{|t|\geq1}|t|^{-1}\langle x-t\rangle^{-1/2}dt\leq I_1'+I_2'+I_3'
\end{equation*}
where
\begin{equation*}
I_1'=\int_{|t|\geq\max\{2|x|,1\}}|t|^{-1}\langle x-t\rangle^{-1/2}dt\leq\int_{|t|\geq\max\{2|x|,1\}}|t|^{-1}\left|\frac{t}{2}\right|^{-1/2}dt\leq 4\sqrt{2}
\end{equation*}
\begin{equation*}
I_2'=\int_{|x|/2\leq|t|\leq2|x|}|t|^{-1}\langle x-t\rangle^{-1/2}dt
\leq\int_{|x|/2\leq|t|\leq2|x|}\frac{2}{|x|}dt=6
\end{equation*}
\begin{equation*}
I_3'=\int_{1\leq|t|\leq|x|/2}|t|^{-1}\langle x-t\rangle^{-1/2}dt
\leq\int_{1\leq|t|\leq|x|/2}|t|^{-3/2}dt\leq2\int_1^{\infty}t^{-3/2}dt
=4.
\end{equation*}
Therefore we have 
\begin{equation}
\label{e:example2}
\|\mathcal{H}(f)'\|_{L^\infty}=\|\mathcal{H}_0(f')\|_{L^\infty}\leq 20.
\end{equation}
\end{exa}

\subsection{Hardy space and outer functions}
\label{s:hardy}


We recall the definition of the Hardy space on the real line
\begin{equation*}
H^2=H^2(\mathbb{R})=\{f\in L^2(\mathbb{R}):\supp\widehat{f}\subset[0,\infty)\}.
\end{equation*}


If $f\in L^2(\mathbb{R})$, then $f+i\mathcal{H}_0(f)\in H^2(\mathbb{R})$ since
by \eqref{e:def-hilbert-alt}, we have 
$$\mathcal{F}(f+i\mathcal{H}_0(f))(\xi)=2\cdot\mathbf{1}_{[0,\infty)}\widehat{f}(\xi).$$


The space of the modulus of functions in $H^2$ can be characterized by the logarithmic integral: for $\omega\in L^2$, $\omega\geq0$, we define
\begin{equation}
\label{e:logint}
\mathcal{L}(\omega):=\int_\mathbb{R}\frac{\log\omega(x)}{1+x^2}dx.
\end{equation}

\begin{thm}
If $f\in H^2$, and $\mathcal{L}(|f|)=-\infty$, then $f\equiv0$. On the other hand, if $\omega\in L^2$, and $\mathcal{L}(\omega)>-\infty$, then there exists a function $f\in H^2$ with $|f|=\omega$, unique up to multiplication by a complex constant with unit modulus.
\end{thm}

For the proof, we refer to \cite[\S 1.5]{up}. For the second statement, we have the following construction of $f$: Let $\Omega=-\log\omega$, then $\Omega\in L^1(\mathbb{R};\langle x\rangle^{-2}dx)$. Therefore we can define $\widetilde{\Omega}=\mathcal{H}(\Omega)$ and take
\begin{equation}
\label{e:outer}
f=ae^{-(\Omega+i\widetilde{\Omega})},\quad |a|=1.
\end{equation}
We call functions of the form \eqref{e:outer} for general $\Omega\in L^1(\mathbb{R};\langle x\rangle^{-2}dx)$ \emph{outer functions}. The class of outer functions contains all nonzero functions in $H^2$, but also other functions that are not in $L^2$. The class of outer function is also closed under multiplication: if $f$ and $g$ are both outer functions, then so is $fg$. Moreover, if two outer functions $f$ and $g$ have the same modulus, i.e. $|f|=|g|$ everywhere, then there is a constant $c\in\mathbb{C}$ with $|c|=1$ such that $f=cg$.


\section{An effective multiplier theorem}
\label{s:effbmm}


In this section, we prove an effective multiplier theorem, which is actually weaker than the Beurling--Malliavin Theorem \ref{t:bmm}. It is essentially \cite[Theorem 1]{bm7}, but we follow the proof there carefully to give an explicit lower bound on $\psi$.

We start with the following lemma from \cite[\S 1.9]{bm7} which gives a sufficient condition for a function to be modulus of the Fourier transform of a function supported in $[0,\sigma]$. 

\begin{lem}
\label{p:modulus}
Assume that $\omega=e^{-\Omega}\in L^2$ and $\mathcal{L}(\omega)>-\infty$. In addition, we assume that $\omega^2e^{2\pi i\sigma x}$ is an outer function. Then there exists $\psi\in L^2$ with $\supp\psi\subset[0,\sigma]$ and $|\widehat{\psi}|=\omega$.
\end{lem}
\begin{proof}
First, we choose $f\in H^2$ such that $|f|=\omega$, then $f^2$ is an outer function with $|f^2|=\omega^2$. But by our assumption so is $\omega^2e^{2\pi i\sigma x}$. Therefore there is a constant $c$ with $|c|=1$ such that $f^2=c\omega^2e^{2\pi i\sigma x}$. On the other hand, $\omega^2=f\bar{f}$, so we have $f=c\bar{f}e^{2\pi i\sigma x}$ and thus $\bar{f}e^{2\pi i\sigma x}\in H^2$. This implies that $\supp\widehat{f}$ is contained in $(-\infty,\sigma]$ and thus in $[0,\sigma]$. We can now take $\psi(\xi)=(\mathcal{F}^{-1}f)(\xi+\sigma)$.
\end{proof}


Now we can state our main theorem in this section.

\begin{thm}
\label{t:effect-multiplier}
Assume that $0<\omega\leq 1$ satisfies $\mathcal{L}(\omega)>-\infty$, and
\begin{equation}
\label{e:hilbert-bd}
\|(\mathcal{H}(\Omega))'\|_{L^\infty}\leq \frac{\pi}{2}\sigma,
\end{equation}
where $0<\sigma<1/10$, $\Omega=-\log\omega$ and $\mathcal{H}$ is the Hilbert transform defined in \eqref{e:hilbert}. Then there is $\psi\in L^2(\mathbb{R})$ with
\begin{equation}
\label{e:psi-supp}
\supp\psi\subset[0,\sigma],
\quad |\widehat{\psi}|\leq\omega,
\end{equation}
and 
\begin{equation}
\label{e:psi-lb}
\|\widehat{\psi}\|_{L^2[-1,1]}\geq \frac{\sigma^6}{600000}\min(\|\omega\|_{L^2(1/2,1)},\|\omega\|_{L^2(-1,-1/2)}).
\end{equation}
\end{thm}


\begin{proof}
We follow the proof in \cite{bm7}. The idea is to find some $\widetilde{\omega}$ satisfying the assumption of Lemma \ref{p:modulus} such that $0\leq\widetilde{\omega}\leq\omega$. In particular, we need $\widetilde{\omega}^2e^{2\pi i\sigma x}$ to be an outer function, in other word, for some constant $c$,
\begin{equation*}
\mathcal{H}(2\log\widetilde{\omega})\equiv 2\pi\sigma x+c \mod{2\pi}.
\end{equation*}
If we write $\widetilde{\omega}=\widetilde{m}\omega$ with $0\leq\widetilde{m}\leq1$, then we need to solve
\begin{equation*}
2\mathcal{H}(\log\widetilde{m})\equiv 2\pi\sigma x+c-2\mathcal{H}(\log\omega)
\mod{2\pi}
\end{equation*}
or equivalently,
\begin{equation}
\label{e:outer-condition}
\mathcal{H}(\log\widetilde{m})\equiv \pi\sigma x+\frac{c}{2}-\mathcal{H}(\log\omega)
\mod{\pi}.
\end{equation}
To obtain $\log\widetilde{m}$, we hope to use the inversion formula \eqref{e:inversion} which requires to take Hilbert transform of $\mathcal{H}(\log\widetilde{m})$. The way to make this Hilbert transform meaningful is to modify the right-hand side of \eqref{e:outer-condition} by constant multiples of $\pi$ to make it bounded. For example, if $\omega\equiv1$, then $\mathcal{H}(\log\omega)\equiv0$ and we can take the sawtooth function (see Figure \ref{f:sawtooth})
\begin{equation}
\label{e:sawtooth}
s(x)=\pi\sigma x-\pi\lfloor\sigma x\rfloor-\frac{\pi}{2}.
\end{equation}
Then $\log\widetilde{m}=-\mathcal{H}(s)$ solves \eqref{e:outer-condition} when $\omega\equiv1$.

\begin{figure}
\centering
\includegraphics[scale=1.1]{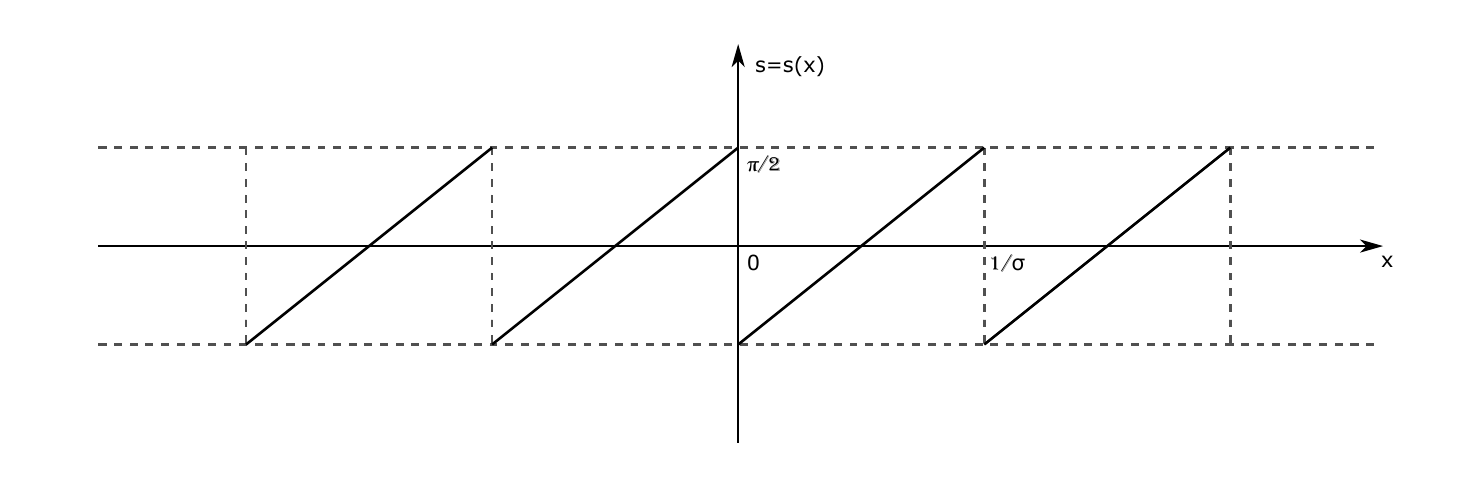}
\caption{The Sawtooth Function.}
\label{f:sawtooth}
\end{figure}

In general, we also need to make sure that $\widetilde{\omega}=\widetilde{m}\omega\in L^2$. However, in the construction, $\widetilde{m}$ may grow polynomially while $\omega$ is merely bounded. To handle this issue, we start the argument above with $\omega$ replaced by some function which already decays polynomially so that the $\widetilde{\omega}$ we obtain is in $L^2$.

To be more precise, we first set 
\begin{equation*}
\omega_0(x)=\frac{\omega(x)}{(x^2+T^2)^3}, \quad \Omega_0=-\log\omega_0
\end{equation*}
where $T$ will be chosen in a moment. Then
\begin{equation*}
\Omega_0=\Omega+3\log(x^2+T^2)
\end{equation*}
and by \eqref{e:example1}, we can compute
\begin{equation*}
\mathcal{H}(\log(x^2+T^2))'=\mathcal{H}(\log((\frac{x}{T})^2+1))'
=T^{-1}\mathcal{H}(\log(x^2+1))'(\cdot/T)=\frac{-2T}{x^2+T^2}.
\end{equation*}
Therefore if we choose $T=\frac{12}{\pi\sigma}$, then
\begin{equation}
\label{e:hilbert-bd2}
\|\mathcal{H}(\Omega_0)'\|_{L^\infty}\leq\|\mathcal{H}(\Omega)'\|_{L^\infty}
+3\|\mathcal{H}(\log(x^2+T^2))'\|_{L^\infty}\leq \pi\sigma.
\end{equation}

\begin{figure}
\centering
\includegraphics[scale=1.1]{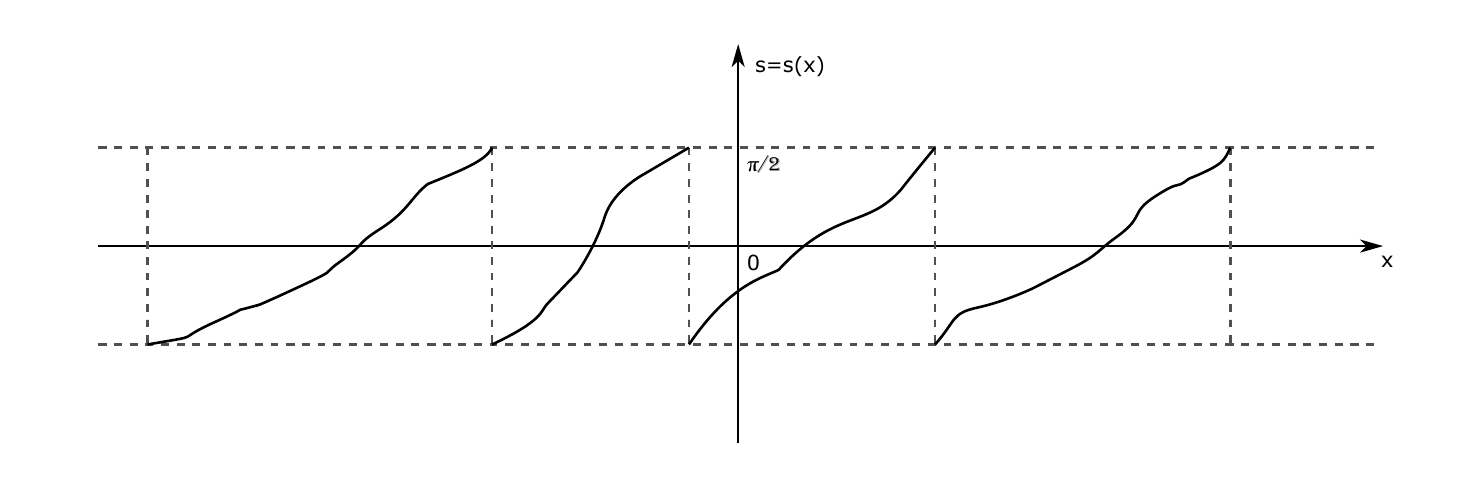}
\caption{The Function $s$.}
\label{f:sawtooth2}
\end{figure}

Let $m=e^{-M}$ where $M=\mathcal{H}(s)$ and $s$ is the function defined by
\begin{equation*}
s(x)=s_0(x)-\pi k(x)-\frac{\pi}{2}
\end{equation*}
(see Figure \ref{f:sawtooth2}) where
\begin{equation*}
s_0(x)=\pi\sigma x+\mathcal{H}(\Omega_0)(x),
\quad
k(x)=\lfloor\pi^{-1}s_0(x)\rfloor.
\end{equation*}
We remark that $s_0$ is monotone increasing and so is $k$. Moreover, $k$ only takes integer values. Now we have $s\in L^\infty$ and thus in $L^1(\mathbb{R},\langle x\rangle^{-2}dx)$. Moreover, $\|s\|_{L^\infty}\leq\frac{\pi}{2}$. By definition \eqref{e:hilbert}
of $\mathcal{H}$, we can write $M(x)=J_1+J_2+J_3$ where
\begin{equation*}
J_1=\frac{1}{\pi}\int_{|x-t|<1/2}\frac{s(t)-s(x)}{x-t}dt;
\end{equation*}
\begin{equation*}
J_2=\frac{1}{\pi}\int_{|x-t|<1/2}s(t)\frac{t}{t^2+1}dt;
\end{equation*}
and
\begin{equation*}
J_3=\frac{1}{\pi}\int_{|x-t|\geq1/2}s(t)\left(\frac{1}{x-t}+\frac{t}{t^2+1}\right)dt.
\end{equation*}

Since for all $t\in\mathbb{R}$, $\left|\frac{t}{t^2+1}\right|\leq 1/2$, we have the bound for $J_2$:
\begin{equation}
\label{e:j2}
|J_2|\leq\frac{1}{\pi}\cdot\|s\|_{L^\infty}\cdot\frac{1}{2}\leq\frac{1}{4}.
\end{equation}

Also, we have
\begin{equation*}
|J_3|\leq\frac{1}{\pi}\|s\|_{L^\infty}\int_{|x-t|\geq1/2}
\left|\frac{1}{x-t}+\frac{t}{t^2+1}\right|dt.
\end{equation*}
Here the integral can be computed explicitly by
\begin{equation*}
\begin{split}
\int_{|x-t|\geq1/2}
\left|\frac{1}{x-t}+\frac{t}{t^2+1}\right|dt
=&\;\frac{1}{2}\log((2x+1)^2+4)+\frac{1}{2}\log((2x-1)^2+4)+\log(x^2+1)\\
\leq&\; 6\log(|x|+2).
\end{split}
\end{equation*}
Therefore we have the following bound on $J_3$:
\begin{equation}
\label{e:j3}
|J_3|\leq 3\log(|x|+2).
\end{equation}

Finally, we need to bound $|J_1|$, at least for some part of $|x|\leq 1$. By \eqref{e:hilbert-bd2}, $s_0(x)=\pi\sigma x+\mathcal{H}(\Omega_0)(x)$ is monotone increasing with $\|s_0'\|_{L^\infty}\leq 2\pi\sigma$. In particular, as we assume that $\sigma<\frac{1}{10}$, we see that $\|s_0'\|_{L^\infty}<\frac{\pi}{4}$ and thus the monotone increasing function $k(x)=\lfloor \pi^{-1}s_0(x)\rfloor$ can take at most two different values on $[-2,2]$. Therefore $k$ is either costant on $[0,2]$ or on $[-2,0]$. Without loss of generality, we assume that $k(x)$ is constant on $[0,2]$, then for $x\in(0,2)$, $s$ is Lipschitz with
\begin{equation*}
\|s'\|_{L^\infty}\leq \pi\sigma+\|\mathcal{H}(\Omega_0)'\|_{L^\infty}
\leq 2\pi\sigma
\end{equation*}
Therefore for $x\in(1/2,1)$, 
\begin{equation}
\label{e:j1-bd-1}
|J_1|\leq\frac{1}{\pi}\int_{|x-t|<1/2}\left|\frac{s(t)-s(x)}{x-t}\right|dt
\leq\frac{1}{\pi}\|s'\|_{L^\infty}\leq 2\sigma.
\end{equation}
For all $x$, we only have a lower bound on $J_1$. Since $s_0(x)=\pi\sigma x+\mathcal{H}(\Omega_0)(x)$ is monotone increasing, $k(x)$ is also monotone increasing. Therefore
\begin{equation}
\label{e:j1-bd-2}
J_1\geq\frac{1}{\pi}\int_{|x-t|<1/2}\frac{s_0(t)-s_0(x)}{x-t}dt
\geq-\frac{1}{\pi}(\pi\sigma+\|\mathcal{H}(\Omega_0)'\|_{L^\infty})
\geq-2\sigma.
\end{equation}

Now combining \eqref{e:j2},\eqref{e:j3} and \eqref{e:j1-bd-1},
we have the following estimates on $M$ on $(1/2,1)$:
\begin{equation}
\label{e:mbd-1}
|M(x)|\leq 2\sigma+\frac{1}{4}+3\log3\leq 5.
\end{equation}
and using \eqref{e:j1-bd-2}, we have for all $x$,
\begin{equation}
\label{e:mbd-2}
M(x)\geq -2\sigma-\frac{1}{2}-3\log(|x|+2)\geq-1-3\log(|x|+2).
\end{equation}

Recall that $m=e^{-M}$, We shall apply Lemma \ref{p:modulus} to the function $\widetilde{\omega}=\frac{1}{3}m\omega_0$. We need to check that $\widetilde{\omega}$ satisfies all the assumptions there. First, we note that by \eqref{e:mbd-2},
\begin{equation*}
0\leq \widetilde{\omega}\leq\frac{e}{3}(|x|+2)^3\omega_0\leq\frac{\omega}{x^2+T^2},
\end{equation*} 
so $0\leq\widetilde{\omega}\leq\omega$ and $\widetilde{\omega}\in L^2$. Moreover, \eqref{e:mbd-2} also implies that
\begin{equation*}
\mathcal{L}(\widetilde{\omega})=\mathcal{L}(m/3)+\mathcal{L}(\omega_0)>-\infty.
\end{equation*}
From the construction of $M=\mathcal{H}(s)$ and the inversion formula \eqref{e:inversion}, we see that
\begin{equation*}
\mathcal{H}(-2M-2\Omega_0)
=2s-2\mathcal{H}(\Omega_0)-2c(M)
=2\pi\sigma x-2\pi k(x)-\pi-2c(M)
\end{equation*}
where $k(x)$ always takes integer values (so that $e^{2\pi ik(x)}\equiv1$) and $c(M)$ is a constant. Therefore for some constant $a$ with $|a|=1$,
\begin{equation*}
\widetilde{\omega}^2e^{2\pi i\sigma x}=\frac{1}{9}e^{-2M-2\Omega_0+2\pi i\sigma x}=\frac{a}{9}e^{-2M-2\Omega_0+i\mathcal{H}(-2M-2\Omega_0)}
\end{equation*}
which shows that $\widetilde{\omega}^2e^{2\pi i\sigma x}$ is an outer function.

Now by Lemma \ref{p:modulus}, there exists $\psi\in L^2$ supported in $[0,\sigma]$ with $|\widehat{\psi}|=\widetilde{\omega}\leq\omega$. Moreover, on $(1/2,1)$, by \eqref{e:mbd-1}, and recall that $T=\frac{12}{\pi\sigma}$, we have
\begin{equation*}
|\widehat{\psi}(x)|=\widetilde{\omega}(x)\geq \frac{1}{3}(1+T^2)^{-3}e^{-5}\omega(x)
\geq\frac{\sigma^6}{600000}\omega(x)
\end{equation*}
which gives a lower bound
\begin{equation*}
\|\widehat{\psi}\|_{L^2(-1,1)}\geq \frac{\sigma^6}{600000}\|\omega\|_{L^2(1/2,1)}.
\end{equation*}
This finishes the proof.
\end{proof}


\section{Proof of the main theorem}
\label{s:proof}


Now we follow the argument in Bourgain-Dyatlov \cite[Section 3]{fullgap} with careful treatment of all the constants to finish the proof of Theorem \ref{t:effup}. We recall the notation that $K$ denotes a universal constant which is often large and varies from line to line.


\subsection{A multiplier adapted to the regular sets}


We first prove an effective version of \cite[Lemma 3.1]{fullgap} which gives a multiplier adapted to regular sets, with the help of Theorem \ref{t:effect-multiplier}.

\begin{lem}
\label{p:adapt}
Assume that $Y\subset[-\alpha_1,\alpha_1]$ is a $\delta$-regular set with constant $C_R$ on scales 2 to $\alpha_1$ and $\delta\in(0,1)$. Fix $0<c_1<1$, then there exist a function $\psi\in L^2(\mathbb{R})$ such that
\begin{equation*}
\supp\psi\subset\left[-\frac{c_1}{10},\frac{c_1}{10}\right],
\end{equation*} 
\begin{equation*}
\|\widehat{\psi}\|_{L^2([-1,1])}\geq c_2,
\end{equation*}
\begin{equation*}
|\widehat{\psi}(\xi)|\leq\exp(-c_3\langle\xi\rangle^{1/2}),
\quad \forall\xi\in\mathbb{R}
\end{equation*}
\begin{equation*}
|\widehat{\psi}(\xi)|\leq\exp(-c_3(\log|\xi|)^{-(1+\delta)/2}|\xi|),
\quad \forall\xi\in Y, |\xi|\geq 10,
\end{equation*}
with
\begin{equation*}
c_2=\frac{c_1^6}{K},
\quad c_3=\frac{c_1\delta(1-\delta)}{KC_R^2}.
\end{equation*}
\end{lem}


\begin{proof}
Recall that in the proof of \cite[Lemma 3.1]{fullgap}, we construct a specific weight $\omega$ adapted to the regular set $Y$. By Proposition \ref{p:smallcover}, we first cover $Y$ by nonoverlapping intervals: Let $A_n=[-2^{n+1},-2^n]\cup[2^n,2^{n+1}]$, then we have a collection $\mathcal{J}_n$ of $N_n$ intervals of size $\rho_n:=n^{-\frac{1+\delta}{2}}2^n$ such that each elements is of the form $\rho_n[j,j+1]$, $j\in\mathbb{Z}$, intersects $A_n$ and
\begin{equation*}
Y\cap A_n\subset\bigcup_{J\in\mathcal{J}_n}J.
\end{equation*}
Moreover, the number $N_n$ of the intervals in $\mathcal{J}_n$ satisfies
\begin{equation*}
N_n\leq 24C_R^2\left(\frac{2^n}{\rho_n}\right)^\delta=24C_R^2n^{\delta(1+\delta)/2}.
\end{equation*}

We fix a cutoff function $\chi(\xi)\in C^\infty(\mathbb{R};[0,1])$ such that
\begin{equation*}
\supp\chi\subset[-1,1], 
\quad
\sup|\partial_\xi\chi|\leq 10,
\quad
\chi=1 \text{ on } [-1/2,1/2].
\end{equation*}
For an interval $J$ with center $\xi_J$, define the cutoff function $\chi_J\in C^\infty(\mathbb{R})$ by rescaling $\chi$:
\begin{equation*}
\chi_J(\xi):=(|J|T_{\xi_J}D_{|J|}\chi)(\xi)=|J|\chi\left(\frac{\xi-\xi_J}{|J|}\right).
\end{equation*}
Then $\chi_J$ satisfies the following 
\begin{equation*}
\supp\chi_J\subset\widetilde{J}=\xi_J+[-|J|,|J|],
\quad
\chi_J=|J| \text{ on } J,
\end{equation*}
\begin{equation*}
0\leq\chi_J\leq|J|,
\quad \sup|\partial_\xi\chi_J|\leq10.
\end{equation*}
Now we define $\omega\in C^\infty(\mathbb{R};(0,1))$ by
\begin{equation*}
\omega(\xi)=\exp(-\langle\xi\rangle^{1/2})\prod_{n=1}^{n_1}\prod_{J\in\mathcal{J}_n}\exp(-10\chi_J),
\end{equation*}
and $\Omega:=-\log\omega=\Omega_1+10\Omega_2$ where $\Omega_1(\xi)=\langle\xi\rangle^{1/2}$, and
\begin{equation*}
\Omega_2=\sum_{n=1}^{n_1}\sum_{J\in\mathcal{J}_n}\chi_J.
\end{equation*}

First, as in \cite[Lemma 3.1]{fullgap}, we have 
\begin{equation*}
\mathcal{L}(\omega)=-\int\frac{\Omega(x)}{1+x^2}dx>-\infty
\end{equation*}

We can bound $\|\mathcal{H}(\Omega)'\|_{L^\infty}$ uniformly in $\alpha_1$. First, by \eqref{e:example2}, we see that 
\begin{equation}
\label{e:bound-omeag-1}
\|\mathcal{H}(\Omega_1)'\|_{L^\infty}\leq 20.
\end{equation}
For
\begin{equation*}
\mathcal{H}(\Omega_2)'=\mathcal{H}_0(\Omega_2')
=\sum_{n=1}^{n_1}\sum_{J\in\mathcal{J}_n}\mathcal{H}_0(\chi_J'),
\end{equation*}
we have by \eqref{e:hilbert-scale}, $\mathcal{H}_0(\chi_J')=\widetilde{\chi}\left(\frac{\cdot-\xi_J}{|J|}\right)$, where $\widetilde{\chi}(x)=\mathcal{H}_0(\chi)'(x)$ is smooth and by \eqref{e:hilbert-deriative-bound},
satisfies the bound
\begin{equation}
\label{e:bound-chi-tilde}
|\widetilde{\chi}(x)|\leq\frac{K}{1+x^2}.
\end{equation}

Now assume that $\xi\in[2^m,2^{m+1}]$ for some $m\geq1$. We estimate
\begin{equation*}
\mathcal{H}(\Omega_2)'(\xi)=\sum_{n=1}^{n_1}\sum_{J\in\mathcal{J}_n}\widetilde{\chi}\left(\frac{\xi-\xi_J}{|J|}\right),
\end{equation*}
by considering different cases for $n$.

When $n\geq m+3$, for each $J\in\mathcal{J}_n$, 
\begin{equation*}
|\xi-\xi_J|\geq 2^n-\xi-\frac{1}{2}\rho_n\geq 2^{n-1}-\xi\geq 2^{n-2}=\frac{1}{4}n^{\frac{1+\delta}{2}}|J|,
\end{equation*}
thus by \eqref{e:bound-chi-tilde}, 
\begin{equation*}
\left|\widetilde{\chi}\left(\frac{\xi-\xi_J}{|J|}\right)\right|\leq Kn^{-1-\delta}.
\end{equation*}
As there are totally $N_n$ intervals in $\mathcal{J}_n$, 
\begin{equation}
\label{e:bound-mnfar}
\sum_{J\in\mathcal{J}_n}\left|\widetilde{\chi}\left(\frac{\xi-\xi_J}{|J|}\right)\right|
\leq KN_nn^{-1-\delta}
\leq KC_R^2n^{-1-\frac{1}{2}(\delta-\delta^2)}.
\end{equation}

When $n\leq m-2$, for each $J\in\mathcal{J}_n$,
\begin{equation*}
|\xi-\xi_J|\geq \xi-2^{n+1}-\frac{1}{2}\rho_n\geq 2^{n+1}-\xi\geq 2^n=n^{\frac{1+\delta}{2}}|J|,
\end{equation*}
thus we have \eqref{e:bound-mnfar} again, as above.

Finally, for each individual $n$, since all $J\in\mathcal{J}_n$ are non-overlapping and have the same length, the points $\frac{\xi-\xi_J}{|J|}$ for $J\in\mathcal{J}_n$ are separated from each other by distance at least 1. Therefore
we have
\begin{equation}
\label{e:bound-mnclose}
\sum_{J\in\mathcal{J}_n}\left|\widetilde{\chi}\left(\frac{\xi-\xi_J}{|J|}\right)\right|\leq2\sum_{n=0}^\infty\frac{K}{1+n^2}\leq K.
\end{equation}
Combining \eqref{e:bound-mnfar} and \eqref{e:bound-mnclose}, we get
\begin{equation}
\label{e:bound-omega-2}
|\mathcal{H}(\Omega_2)'|\leq K+KC_R^2\sum_{n=1}^\infty n^{-1-\frac{1}{2}(\delta-\delta^2)}\leq\frac{KC_R^2}{\delta(1-\delta)}.
\end{equation}
Now we have the uniform bound for $\mathcal{H}(\Omega)'$ from \eqref{e:bound-omeag-1} and \eqref{e:bound-omega-2}
\begin{equation*}
\|\mathcal{H}(\Omega)'\|_{L^\infty}\leq \frac{KC_R^2}{\delta(1-\delta)}.
\end{equation*}
For $\xi\in[-2,2]$, we can get the same upper bound similarly. Now we can apply Theorem \ref{t:effect-multiplier} to $\omega^{c_3}$ to conclude the proof.
\end{proof}


\subsection{The iterative step}

We also need the following bound on functions with compactly supported Fourier transform \cite[Lemma 3.2]{fullgap} but with explicit constants. 


\begin{lem}
\label{p:boundcompact}
Assume that $\mathcal{I}$ is a nonoverlapping collection of intervals of size 1 each, and for each $I\in\mathcal{I}$ we choose a subinterval $I''$ with $|I''|=c_0>0$ independent of $I$. Then for all $r\in(0,1)$, $\kappa\in(0, \frac{c_0}{8}e^{-2/r})$, and $f\in L^2(\mathbb{R})$ with $\widehat{f}$ compactly supported, we have
\begin{equation}
\label{e:boundcompact}
\sum_{I\in\mathcal{I}}\|f\|_{L^2(I)}^2
\leq \frac{K}{c_0r}\left(\sum_{I\in\mathcal{I}}\|f\|_{L^2(I'')}^2\right)^\kappa\|e^{2\pi r|\xi|}\widehat{f}(\xi)\|_{L^2}^{2(1-\kappa)}.
\end{equation}
\end{lem}

We remark that in \cite[Lemma 3.2]{fullgap}, we have $C/r$ instead of $K/c_0r$ on the right-hand side of \eqref{e:boundcompact} with $C$ depending only on $c_0$. The dependence of $C$ on $c_0$ comes from \cite[(3.16)]{fullgap} and is given by \cite[Lemma 2.13]{fullgap}:
\begin{equation*}
C=\left(\int_0^{c_0}(x(c_0-x))^{-2/3}dx\right)^3\leq K/c_0.
\end{equation*}


Now let $\mathcal{I}=\{[j,j+1]\colon j\in\mathbb{Z}\}$ and for each $I\in\mathcal{I}$, we are given a subinterval $I'\subset I$ with
$|I'|=c_1>0$ independent of $I$. Define 
\begin{equation*}
U'=\bigcup_{I\in\mathcal{I}}I'.
\end{equation*}

As in \cite[\S 3.3]{fullgap}, using the previous two lemmas, we obtain the following lemma (see \cite[Lemma 3.4]{fullgap}). We note that in \cite{fullgap}, the notation is slightly different: the letter $K$ (instead of $M$) is used to denote the frequency separating the low frequency part and the high frequency part. However, we use $M$ since $K$ is always a universal constant in our notation.


\begin{lem}
\label{p:iterative-1}
Assume that $Z\subset[-\alpha_1,\alpha_1]$ is $\delta$-regular with constant $C_R$ on scales 1 to $\alpha_1\geq2$. Then for all $f\in L^2(\mathbb{R})$ with $\supp\widehat{f}\subset Z$, all $M>10$, and 
\begin{equation*}
\kappa=\exp(-C_1(\log M)^{\frac{1+\delta}{2}}),
\end{equation*}
we have
\begin{equation*}
\|\widehat{f}\|_{L^2(-1,1)}^2
\leq C_2M^{21}
\left(\|1_{U'}f\|_{H^{-10}}^2+\exp(-C_3^{-1}(\log M)^{-\frac{1+\delta}{2}}M)\|f\|_{H^{-10}}^2\right)^{\kappa}\|f\|_{H^{-10}}^{2(1-\kappa)}.
\end{equation*}
Here
\begin{equation}
\label{e:consts}
C_1=\frac{KC_R^2(1+\log(1/c_1))}{c_1\delta(1-\delta)}，
\quad
C_2=\frac{KC_R^{62}}{c_1^{44}\delta^{31}(1-\delta)^{31}},
\quad
C_3=\frac{KC_R^2}{c_1\delta(1-\delta)}.
\end{equation}
\end{lem}


\begin{proof}
For each $I\in\mathcal{I}$, let $I''\subset I'$ be the interval with the same center as $I'$ and $$|I''|=\frac{1}{2}|I'|=\frac{1}{2}c_1.$$ 
Denote $$U''=\bigcup_{I\in\mathcal{I}}I''.$$

Apply Lemma \ref{p:adapt} to $Y=Z(2):=Z+[-2,2]\subset[-(\alpha_1+2),\alpha_1+2]$ which is a $\delta$-regular set with constant $100C_R$ on scales 2 to $\alpha_1+2$:
\begin{equation}
\label{e:psi-supp-1}
\supp\psi\subset\left[-\frac{c_1}{10},\frac{c_1}{10}\right]
\end{equation}
\begin{equation}
\label{e:psi-lb-1}
\|\widehat{\psi}\|_{L^2(-1,1)}\geq c_2
\end{equation}
\begin{equation}
\label{e:psi-uba}
|\widehat{\psi}(\xi)|\leq\exp(-c_3\langle\xi\rangle^{1/2}),
\quad \forall\xi\in\mathbb{R}
\end{equation}
\begin{equation}
\label{e:psi-uby}
|\widehat{\psi}(\xi)|\leq\exp(-c_3(\log|\xi|)^{-(1+\delta)/2}|\xi|),
\quad \forall\xi\in Y, |\xi|\geq 10,
\end{equation}
with
\begin{equation*}
c_2=\frac{c_1^6}{K},
\quad
c_3=\frac{c_1\delta(1-\delta)}{KC_R^2}.
\end{equation*}

For $\eta\in[-2,2]$, we define $f_\eta(x)=e^{2\pi i\eta x}f(x)$ so that
$\widehat{f}_\eta(\xi)=\widehat{f}(\xi-\eta)$. We put
\begin{equation*}
g_\eta=f_\eta\ast\psi\in L^2(\mathbb{R}).
\end{equation*}

By \eqref{e:psi-supp-1}, we see that on $U''$, $g_\eta=(1_{U'}f_\eta)\ast\psi$
where $1_{U'}f_\eta(x)=e^{2\pi i\eta x}1_{U'}f(x)$, so
\begin{equation*}
\|g_\eta\|_{L^2(U'')}\leq \|(1_{U'}f_\eta)\ast\psi\|_{L^2}
=\|\widehat{1_{U'}f_\eta}(\xi)\widehat{\psi}(\xi)\|_{L^2}
=\|\widehat{1_{U'}f}(\xi-\eta)\widehat{\psi}(\xi)\|_{L^2}.
\end{equation*}
Therefore using
\begin{equation*}
\sup_{\xi\in\mathbb{R},|\eta|\leq2}|
\langle\xi\rangle^{10}\widehat{\psi}(\xi+\eta)|\leq
\sup_{\xi\in\mathbb{R},|\eta|\leq2}\langle\xi\rangle^{10}
\exp(-c_3\langle\xi+\eta\rangle^{1/2})\leq C_4:=Kc_3^{-20}.
\end{equation*}
we get the $L^2$-bound of $g_\eta$ on $U''$:
\begin{equation}
\label{e:g-eta-u}
\|g_\eta\|_{L^2(U'')}\leq C_4\|1_{U'}f\|_{H^{-10}}.
\end{equation}

On the other hand, since $\supp\widehat{f_\eta}\subset Z(2)$, by \eqref{e:psi-uby}, for all $\xi$,
\begin{equation*}
|\widehat{g}_\eta(\xi)|
\leq
\exp(-c_3(\log(10+|\xi|))^{-\frac{1+\delta}{2}}|\xi|)|\widehat{f}_\eta(\xi)|.
\end{equation*}

We take 
\begin{equation*}
r:=\frac{c_3}{10}(\log M)^{-\frac{1+\delta}{2}}\in(0,1)
\end{equation*}
so that
\begin{equation}
\label{e:bound-lowfreq}
\sup_{|\xi|\leq M}e^{2\pi r|\xi|}
\exp(-c_3(\log(10+|\xi|))^{\frac{1+\delta}{2}}|\xi|)\leq1.
\end{equation}

Now we decompose $g_\eta=g_1+g_2$ where $g_1,g_2\in L^2$ and
\begin{equation*}
\supp\widehat{g}_1\subset\{|\xi|\leq M\},\quad
\supp\widehat{g}_2\subset\{|\xi|\geq M\}.
\end{equation*}
Then by \eqref{e:bound-lowfreq}, we get
\begin{equation}
\label{e:bound-g1-exp}
\|e^{2\pi r|\xi|}\widehat{g}_1(\xi)\|_{L^2}\leq KM^{10}\|f\|_{H^{-10}}.
\end{equation}
For $g_2$, we can use
\begin{equation*}
\sup_{|\eta|\leq2}\sup_{|\xi|\geq M}
\exp(-c_3(\log(10+|\xi))^{-\frac{1+\delta}{2}}|\xi|)\langle\xi-\eta\rangle
\leq C_5\exp(-C_6^{-1}(\log M)^{-\frac{1+\delta}{2}}M)
\end{equation*}
with $C_5=K/c_3^{11},\quad C_6=K/c_3$ to get
\begin{equation}
\label{e:g2-b}
\|g_2\|_{L^2}\leq C_5\exp(-C_6^{-1}(\log M)^{-\frac{1+\delta}{2}}M)\|f\|_{H^{-10}}.
\end{equation}

Now we apply Lemma \ref{p:boundcompact} to $g_1$, so
\begin{equation*}
\|g_1\|_{L^2}^2\leq\frac{K}{c_1r}\|g_1\|_{L^2(U'')}^{2\kappa}\|e^{2\pi r|\xi|}\widehat{g}_1(\xi)\|_{L^2}^{2(1-\kappa)}
\leq \frac{KM^{20}}{c_1r}\|g_1\|_{L^2(U'')}^{2\kappa}\|f\|_{H^{-10}}^{2(1-\kappa)}
\end{equation*}
where
\begin{equation*}
\kappa:=\exp(-C_7/r)=\exp(-C_8(\log K)^{\frac{1+\delta}{2}})<\frac{c_1}{16}e^{-2/r},
\end{equation*}
and $C_7=6-\log c_1$, $C_8=10C_7/c_3$. But by \eqref{e:g-eta-u}
\begin{equation*}
\|g_1\|_{L^2(U'')}\leq\|g_\eta\|_{L^2(U'')}+\|g_2\|_{L^2}
\leq C_4\|1_{U''}f\|_{H^{-10}}+\|g_2\|_{L^2}.
\end{equation*}
Therefore
\begin{equation*}
\|g_1\|_{L^2}^2
\leq \frac{KM^{20}}{c_1r} (C_4\|1_{U''}f\|_{H^{-10}}+\|g_2\|_{L^2})^{2\kappa}
\|f\|_{H^{-10}}^{2(1-\kappa)}
\end{equation*}
and by \eqref{e:g2-b}, we have
\begin{equation}
\label{e:g1-b}
\|g_1\|_{L^2}^2
\leq C_9M^{20}r^{-1}(\|1_{U''}f\|_{H^{-10}}+\exp(-C_6^{-1}(\log M)^{-\frac{1+\delta}{2}}M)\|f\|_{H^{-10}})^{2\kappa}\|f\|_{H^{-10}}^{2(1-\kappa)},
\end{equation}
where
\begin{equation*}
C_9=\frac{KC_4C_5}{c_1}=\frac{K}{c_1c_3^{31}},
\quad r^{-1}=\frac{10}{c_3}(\log M)^{\frac{1+\delta}{2}}\leq KM/c_3.
\end{equation*}
Finally, combining \eqref{e:g1-b} and \eqref{e:g2-b}, we get
\begin{equation}
\label{e:g-b}
\|g_\eta\|_{L^2}^2
\leq C_{10}M^{21}(\|1_{U''}f\|_{H^{-10}}+\exp(-C_{11}^{-1}(\log M)^{-\frac{1+\delta}{2}}M)\|f\|_{H^{-10}})^{2\kappa}\|f\|_{H^{-10}}^{2(1-\kappa)}
\end{equation}
where
\begin{equation*}
C_{10}=\frac{K}{c_1c_3^{31}}=\frac{KC_R^{62}}{c_1^{32}\delta^{31}(1-\delta)^{31}},
\quad C_{11}=\frac{KC_R^2}{c_1\delta(1-\delta)}.
\end{equation*}
By \eqref{e:psi-lb-1},
\begin{equation*}
\|\widehat{f}\|_{L^2(-1,1)}^2
\leq c_2^{-2}\int_{-1}^1\int_{-1}^1|\widehat{f}(\zeta)\widehat{\psi}(\xi)|^2d\xi d\zeta
\leq c_2^{-2}\int_{-2}^2\int_{\mathbb{R}}|\widehat{f}(\xi-\eta)\widehat{\psi}(\xi)|^2d\xi d\eta
\end{equation*}
Using 
\begin{equation*}
\widehat{g}_\eta(\xi)=\widehat{f}_\eta(\xi)\widehat{\psi}(\xi)=\widehat{f}(\xi-\eta)\widehat{\psi}(\xi)
\end{equation*}
and \eqref{e:g-b}, we can finish the proof.
\end{proof}


This allows us to get the following quantitative unique continuation result on functions with Fourier support in a regular set.


\begin{thm}
Assume that $Y\subset[-\alpha_1,\alpha_1]$ is $\delta$-regular with constant $C_R$ on scales 1 to $\alpha_1$, and $\delta\in(0,1)$. Then for all $f\in L^2(\mathbb{R})$, with $\supp\widehat{f}\subset Y$,
\begin{equation}
\|f\|_{L^2(U')}\geq c_4\|f\|_{L^2}
\end{equation}
where
\begin{equation}
\label{e:iterimprove}
c_4=\exp\left[-\exp\left(K\left(\frac{C_R^2}{c_1\delta(1-\delta)}\right)^{\frac{K}{1-\delta}}\right)\right].
\end{equation}
\end{thm}

\begin{proof}
For $\ell\in\mathbb{Z}$, we define $f_\ell(x)=e^{2\pi i\ell x}f(x)$, then
$\widehat{f}_\ell(\xi)=\widehat{f}(\xi-\ell)$ and
$\supp\widehat{f}_\ell\subset Y+\ell\subset[-2\alpha_1,2\alpha_1]$. Apply Lemma \ref{p:iterative-1} to $f_\ell$ with $Z=Y+\ell$ which is $\delta$-regular with constant $4C_R$ on scales 1 to $2\alpha_1$, we get for all $M>10$,
and
\begin{equation*}
\kappa=\exp(-C_1(\log M)^{\frac{1+\delta}{2}}),
\end{equation*}
\begin{equation*}
\|\widehat{f}_\ell\|_{L^2(-1,1)}^2
\leq C_2M^{21}
\left(\|1_{U'}f_\ell\|_{H^{-10}}^2+\exp(-C_3^{-1}(\log M)^{-\frac{1+\delta}{2}}M)\|f_\ell\|_{H^{-10}}^2\right)^{\kappa}\|f_\ell\|_{H^{-10}}^{2(1-\kappa)},
\end{equation*}
where the constants $C_1,C_2,C_3$ are given by \eqref{e:consts}.
Since $\supp\widehat{f}\subset Y\subset[-\alpha_1,\alpha_1]$, we see
\begin{equation*}
\|f\|_{L^2}^2=\|\widehat{f}\|_{L^2}^2\leq\sum_{|\ell|\leq\alpha_1}\|\widehat{f}_\ell\|_{L^2(-1,1)}^2.
\end{equation*}
By H\"{o}lder's inequality
\begin{equation*}
\sum a_j^\kappa b_j^{1-\kappa}\leq\left(\sum a_j\right)^{\kappa}\left(\sum b_j\right)^{1-\kappa},
\end{equation*}
we get
\begin{equation*}
\begin{split}
\|f\|_{L^2}^2
\leq &\;C_2M^{21}
\left(\sum_\ell\|1_{U'}f_\ell\|_{H^{-10}}^2+\exp(-C_3^{-1}(\log M)^{-\frac{1+\delta}{2}}M)\sum_\ell\|f_\ell\|_{H^{-10}}^2\right)^{\kappa}\\
&\;\cdot\left(\sum_{\ell}\|f_\ell\|_{H^{-10}}\right)^{2(1-\kappa)}.
\end{split}
\end{equation*}
Now since
\begin{equation*}
\sum_{\ell}\|f_\ell\|_{H^{-10}}^2\leq K\|f\|_{L^2}^2,
\quad
\sum_{\ell}\|1_{U'}f_\ell\|_{H^{-10}}^2\leq K\|f\|_{L^2(U')}^2
\end{equation*}
by Minkowski inequality, $(a+b)^\kappa\leq a^\kappa+b^\kappa$, we have
\begin{equation*}
\|f\|_{L^2}^2
\leq C_2KM^{21}\|f\|_{L^2(U'')}^{2\kappa}\|f\|_{L^2}^{2(1-\kappa)}
+C_2KM^{21}\exp(-C_3^{-1}(\log M)^{-\frac{1+\delta}{2}}M\kappa)\|f\|_{L^2}^2
\end{equation*}
Now by taking $M$ large enough, for example
\begin{equation*}
M=\exp\left[K\left(\frac{C_R^2}{c_1\delta(1-\delta)}\right)^{\frac{3}{1-\delta}}\right]
\end{equation*} 
we can make
\begin{equation*}
C_2KM^{21}\exp(-C_3^{-1}(\log M)^{-\frac{1+\delta}{2}}M\kappa)<\frac{1}{2}
\end{equation*}
then we finish the proof with 
\begin{equation*}
c_4=(2C_2KM^{21})^{-\frac{1}{2\kappa}}
=\exp\left[-\exp\left(K\left(\frac{C_R^2}{c_1\delta(1-\delta)}\right)^{\frac{K}{1-\delta}}\right)\right].
\end{equation*}
\end{proof}

\subsection{Finishing the proof}

Now to finish the proof, we follow \cite[\S 3.4]{fullgap} to iterate this argument $\sim\log N$ times and refer there for the details. In the end, we get
\begin{equation*}
\beta=-\frac{\log(1-\frac{\tau}{2})}{T\log L}.
\end{equation*}
Here $L=\lceil(3C_R)^{\frac{2}{1-\delta}}\rceil$, $\tau=Kc_4^2$, where $c_4$ is given by \eqref{e:iterimprove} with $c_1=(2L)^{-1}$ and finally $T$ is chosen so that 
\begin{equation*}
\left(1-\frac{K}{L^{T-1}}\right)^{-1}(1-\tau)\leq(1-\frac{\tau}{2}).
\end{equation*}
Therefore we have \eqref{e:fupexp} and this finishes the proof of Theorem \ref{t:effup}.


\section{General fractal uncertainty principle and applications}
\label{s:application}


\subsection{A general fractal uncertainty principle}
\label{s:generalfup}


Theorem \ref{t:effup} implies the following fractal uncertainty principle
for general Fourier integral operators as in \cite[\S 4.2]{fullgap}. 

\begin{thm}
\label{t:generalfup}
Let $B=B(h):L^2(\mathbb{R})\to L^2(\mathbb{R})$ be of the form
\begin{equation*}
Bf(x)=h^{-1/2}\int_{\mathbb{R}}e^{i\Phi(x,y)/h}b(x,y)f(y)dy
\end{equation*}
where the amplitude $b\in C_0^\infty(U)$ for some open set $U\subset\mathbb{R}^2$ and the phase $\Phi\in C^\infty(U;\mathbb{R})$ satisfies $\partial_{xy}^2\Phi\neq0$ on $U$. Let $0\leq\delta<1$, $C_R\geq1$ and assume that $X,Y\subset\mathbb{R}$ are $\delta$-regular with constant $C_R$ on scales from 0 to 1. Then there exist $\beta>0$ and $\rho\in(0,1)$ depending only on $\delta$ and $C_R$, such that for all $h\in(0,1)$
\begin{equation}
\label{e:generalfup}
\|1_{X(h^\rho)}B(h)1_{Y(h^\rho)}\|_{L^2(\mathbb{R})\to L^2(\mathbb{R})}\leq Ch^\beta
\end{equation}
where $X(h^\rho)=X+[-h^\rho,h^\rho]$ and similarly for $Y(h^\rho)$. Here the constant $C$ depends only on $\delta, C_R,\Phi,b$ and we can take 
\begin{equation}
\label{e:generalfupexp}
\rho=1-\frac{\beta}{2},
\quad\beta=\exp\left[-K(C_R\delta^{-1}(1-\delta)^{-1})^{K(1-\delta)^{-3}}\right].
\end{equation} 
\end{thm}

We remark that here in \eqref{e:generalfupexp}, we get an additional power of $(1-\delta)^{-1}$ in the triple exponent comparing to \eqref{e:fupexp} because in the proof of \cite[\S 4]{fullgap}, we need to separate the set $X(h^\rho)$ and $Y(h^\rho)$ into many smaller $\delta$-regular sets with constant $C_R'=C_R^{K(1-\delta)^{-1}}$. 


\subsection{Applications to convex co-compact hyperbolic surfaces}
\label{s:hyper}

Let $M=\Gamma\backslash\mathbb H^2$ be a convex co-compact hyperbolic surface,
$\Lambda_\Gamma\subset\mathbb S^1$ be its limit set, $\delta\in [0,1)$ be the dimension
of $\Lambda_\Gamma$, and $\mu$ be the Patterson--Sullivan measure, which is a
probability measure supported on $\Lambda_\Gamma$, see for instance~\cite[\S14.1]{BorthwickBook}. Then $\Lambda_\Gamma$ is a $\delta$-regular set on scales 0 to 1 with some constant $C_R$. Applying Theorem \ref{t:generalfup}, we get the following hyperbolic fractal uncertainty principle for $\Lambda_\Gamma.$

\begin{thm}
  \label{t:hyperfup}
Let $B_\chi=B_\chi(h):L^2(\mathbb{S}^1)\to L^2(\mathbb{S}^1)$ be defined by
$$B_\chi f(y)=(2\pi h)^{-1/2}\int_{\mathbb{S}^1}|y-y'|^{2i/h}\chi(y,y')f(y')dy'$$
where $|y-y'|$ is the Euclidean distance in $\mathbb{R}^2$ restricted to the unit circle $\mathbb{S}^1$ and $\chi\in C_0^\infty(\mathbb{S}^1\times\mathbb{S}^1\setminus\{y=y'\})$. Then for $\beta$ and $\rho$ given by \eqref{e:generalfupexp}, we have for all $\chi$ as above and $h\in(0,1)$,
$$\|1_{\Lambda_\Gamma(h^\rho)}B_\chi(h)1_{\Lambda_\Gamma(h^\rho)}\|_{L^2(\mathbb{S}^1)\to L^2(\mathbb{S}^1)}\leq Ch^\beta.$$
Here $C$ may depend on $\delta$, $C_R$ and $\chi$, but not $h$.
\end{thm}

Then by \cite[Theorem 3]{fup}, we get the following result on essential spectral gaps and cutoff resolvent estimates for the Laplacian on convex co-compact hyperbolic surfaces.

\begin{thm}
  \label{t:gap-hyper}
Consider the meromorphic scattering resolvent
$$
R(\lambda)=\Big(-\Delta_M-\frac{1}{4}-\lambda^2\Big)^{-1}:\begin{cases}
L^2(M)\to L^2(M),\quad \Im\lambda >0,\\
L^2_{\comp}(M)\to L^2_{\loc}(M),\quad \Im\lambda\leq 0.
\end{cases}
$$
Assume that $0<\delta<1$. Then for some universal constant $K>0$, $M$ has an essential spectral gap of size
\begin{equation}
  \label{e:gap-hyper}
\beta_M=\exp\left[-\exp\left(K(C_R\delta^{-1}(1-\delta)^{-1})^{K(1-\delta)^{-3}}\right)\right].
\end{equation}
That is, $R(\lambda)$ has only finitely many poles
in $\{\Im\lambda>-\beta_M\}$. Moreover it satisfies the cutoff estimates:
for each $\psi\in C_0^\infty(M),\varepsilon>0$
and some constant $C_0$ depending on $\varepsilon$,
$$
\|\psi R(\lambda)\psi\|_{L^2\to L^2}\leq C(\psi,\varepsilon)|\lambda|^{-1-2\min(0,\Im\lambda)+\varepsilon},\quad
\Im\lambda \in [-\beta_M,1],\quad
|\Re\lambda|\geq C_0.
$$
\end{thm}

\subsection{Applications to open quantum baker's maps}
\label{s:oqm}

Following the notation in~\cite{oqm} and \cite[Section 5]{regfup}, we consider an open quantum baker's map $B_N$
determined by a triple $(M,\mathcal{A},\chi)$ where $M\in\mathbb{N}$ is called the base, $\mathcal{A}\subset\mathbb{Z}_M=\{0,1,\ldots,M-1\}$ is called the alphabet, and $\chi\in C_0^\infty((0,1);[0,1])$ is a cutoff function.
The map $B_N$ is a sequence of operators 
 $B_N:\ell^2_N\to\ell^2_N$,
 $\ell^2_N=\ell^2(\mathbb Z_N)$, defined for every positive $N\in M\mathbb Z$ by
\begin{equation}
  \label{e:oqm-def}
B_N=\mathcal F_N^*\begin{pmatrix}
\chi_{N/M}\mathcal F_{N/M}\chi_{N/M} & \\
&\ddots&\\
 & &\chi_{N/M}\mathcal F_{N/M}\chi_{N/M}
\end{pmatrix}I_{\mathcal A, M}
\end{equation}
where $\mathcal F_N$ is the unitary Fourier transform given by the
$N\times N$ matrix $\frac{1}{\sqrt{N}}(e^{-2\pi ij\ell/N})_{j\ell}$, $\chi_{N/M}$ is the multiplication operator
on~$\ell^2_{N/M}$ discretizing $\chi$, and
$I_{\mathcal A,M}$ is the diagonal matrix
with $\ell$-th diagonal entry equal to~1 if $\lfloor\frac{l}{N/M}\rfloor\in\mathcal A$
and 0 otherwise. For the background of open quantum baker's map, we refer to \cite{oqm} and the references there.

Let $\delta$ be the dimension of the Cantor set consisting all real numbers in $[0,1]$ with each digit in the base-$M$ expansion an element of $\mathcal{A}$. In particular,
\begin{equation}
\label{e:dim-cantor}
\delta=\frac{\log|\mathcal{A}|}{\log M}.
\end{equation}

Then we have the following theorem (see \cite[Theorem 4]{regfup}).
\begin{thm}
  \label{t:gap-improves}
Assume that $0<\delta<1$, that is $1<|\mathcal A|<M$. Then there exists
\begin{equation}
  \label{e:improved-beta}
\beta=\beta(M,\mathcal A)>\max\Big(0,\frac{1}{2}-\delta\Big)
\end{equation}
such that, with $\Sp(B_N)\subset \{\lambda\in \mathbb C\colon |\lambda|\leq 1\}$ denoting the spectrum of $B_N$ defined by \eqref{e:oqm-def},
\begin{equation*}
\limsup_{N\to\infty,\,N\in M\mathbb Z}\max\{|\lambda|\colon \lambda\in\Sp(B_N)\}\ \leq\ M^{-\beta}.
\end{equation*}
\end{thm}

In particular, we have for $\delta\leq\frac{1}{2}$, as in \cite[\S 5.3]{regfup}, $\beta$ in \eqref{e:improved-beta} is given by
\begin{equation}
\label{e:gap-p}
\beta\geq \frac{1}{2}-\delta+(40M^{3\delta})^{-\frac{160}{\delta(1-\delta)}}
\end{equation}
and for $\delta\geq\frac{1}{2}$, using \eqref{t:effup}, we get for some universal constant $K>0$,
\begin{equation}
\label{e:gap-t}
\beta\geq\exp\left[-\exp\left(KM^{K(1-\delta)^{-2}}\right)\right].
\end{equation}

We can compare our result to \cite{oqm} where we only consider the special sequence $N=M^k$ and get
\begin{equation}
\label{e:gap-p-sp}
\beta\geq \frac{1}{2}-\delta+\frac{1}{KM^8\log M}
\end{equation}
and
\begin{equation}
\label{e:gap-t-sp}
\beta\geq \exp\left[-M^{\frac{\delta}{1-\delta}+o_M(1)}\right].
\end{equation}
The gap for the special sequence \eqref{e:gap-p-sp} and \eqref{e:gap-t-sp} are better than the general case \eqref{e:gap-p} and \eqref{e:gap-t}. In particular, when $\delta\leq\frac{1}{2}$, we get polynomial improvement in $M$ (over the pressure gap $\beta=\frac{1}{2}-\delta$) while when $\delta\geq\frac{1}{2}$, we get exponentially small improvement in $M$ (over the trivial gap $\beta=0$) for special sequence, but only double exponentially small improvement in $M$. However, we remark that when $\delta$ is very close to $1/2$, for the special sequence $N=M^k$, we have a much better improvement in \cite{oqm} using the method of additive energy
\begin{equation}
\label{e:gap-ae}
\beta\geq\max\left(\frac{1}{2}-\delta,0\right)+\frac{1}{K\log M},\quad \text{ for } \quad \left|\delta-\frac{1}{2}\right|\leq\frac{1}{K\log M}.
\end{equation}
We do not have an estimate comparable to \eqref{e:gap-ae} for general $N$.

\def\arXiv#1{\href{http://arxiv.org/abs/#1}{arXiv:#1}}

\end{document}